\documentclass[reqno,12pt]{amsart} 
\usepackage{amssymb, amsmath}
\usepackage{enumerate,mathrsfs,colonequals,bm}
\usepackage{amsthm}
\usepackage{fullpage,url,tikz,xspace,setspace}
\usepackage{microtype}
\usepackage{calc}
\usepackage[unicode=true]{hyperref}
\newtheorem{thm}{Theorem}[section]
\newtheorem{crl}[thm]{Corollary}

\newtheorem{lmm}[thm]{Lemma}

\newtheorem{prp}[thm]{Proposition}

\theoremstyle{definition}
\newtheorem{dfn}[thm]{Definition}
\theoremstyle{Remark}
\newtheorem{rmk}[thm]{Remark}
\newtheorem{exa}[thm]{Example}
\theoremstyle{remark}
\newcommand{\Z}{\mathbb Z}
\newcommand{\FF}{\mathbb F}

\newcommand{\oFp}{\overline{\FF}_p}
\newcommand{\HH}{\mathbb{H}}
\newcommand{\msH}{\mathscr{H}}

\newcommand{\msA}{\mathscr{A}}
\newcommand{\msL}{\mathscr{L}}
\newcommand{\msP}{\mathscr{P}}

\newcommand{\cO}{{\mathcal O}}

\newcommand{\Q}{{\mathbb Q}}
\newcommand{\N}{{\mathbb N}}
\newcommand{\A}{\msA}

\newcommand{\Gr}{\mathit{Gr}}
\newcommand{\wgr}{\widetilde{\mathit{gr}}}
\newcommand{\gr}{\mathit{gr}}

\newcommand{\cB}{\mathcal{B}}

\newcommand{\cSS}{\mathcal{S\hspace*{-1pt}S}}
\newcommand{\GR}{\mathit{Gr}}
\DeclareMathOperator\Ad{Ad}
\DeclareMathOperator\Gal{Gal}

\DeclareMathOperator\w{w}
\DeclareMathOperator\Adw{\Ad_w}
\DeclareMathOperator\iso{iso}
\DeclareMathOperator\Aut{Aut}

\DeclareMathOperator\sS{\sf SP}
\DeclareMathOperator\Pic{Pic}

\newcommand\PGL{{\rm PGL}}
\DeclareMathOperator\End{End}
\newcommand\SL{{\rm SL}}
\newcommand\Sp{{\rm Sp}}
\newcommand\GL{{\rm GL}}
\renewcommand\U{{\rm U}}
\DeclareMathOperator\Iso{Iso}
\DeclareMathOperator\Disc{Disc}
\DeclareMathOperator\Frobb{Frob}
\DeclareMathOperator\Id{Id}
\newcommand\GU{{\rm GU}}
\DeclareMathOperator\Nm{Nm}
\DeclareMathOperator\HNm{HNm}
\DeclareMathOperator\Grr{Gr}
\DeclareMathOperator\Ver{Ver}
\DeclareMathOperator\Ed{Ed}
\DeclareMathOperator\Mat{Mat}

\newcommand{\sSgpz}{\sS_{\!g}(p)_0}
\newcommand{\sSgpg}{\sS_{\!g}(p)_g}

\begin{document}
\title{Isogeny graphs of superspecial abelian varieties}

\author{Bruce~W.~Jordan}
\address{Department of Mathematics, Baruch College, The City University
of New York, One Bernard Baruch Way, New York, NY 10010-5526, USA}
\email{bruce.jordan@baruch.cuny.edu}

\maketitle

\begin{abstract}      
We define three different isogeny graphs of principally polarized
superspecial abelian varieties, prove foundational results on them,
and explain their role 
in number theory and geometry.  This is background  to
joint work with Yevgeny Zaytman on properties of these isogeny graphs
for dimension $g>1$,
especially
the result that they are connected, but not in general Ramanujan.
\end{abstract}

\section{Introduction}

A superspecial abelian variety $A/\oFp$ of dimension $g$
is by definition isomorphic
to a product of $g$ supersingular elliptic curves.  There is in fact only 
{\em one} superspecial abelian variety of dimension $g>1$:  
Fix a supersingular elliptic curve $E/\oFp$ with 
$\cO=\cO_E=\End(E)$ a maximal order in the rational definite quaternion
algebra $\HH_p$ ramified at $p$.
\begin{thm}
\label{DOS}
{\rm (Deligne, Ogus \cite{og}, Shioda \cite{shi}) } Suppose $A/\oFp$ is a superspecial abelian
variety with $\dim A=g>1$.  Then $A\cong E^g$.
\end{thm} 
If $A=E^g$, then 
\[
\End(A)=\Mat_{g\times g}(\cO)\subseteq
\End^{0}(A)=\End(A)\otimes \Q=\Mat_{g\times g}(\HH_p).
\]
The theory of superspecial abelian varieties thus bifurcates:
for dimension $g=1$ there are {\em many} superspecial abelian varieties
(= supersingular elliptic curves) each with {\em one} principal polarization,
whereas for dimension $g>1$ there is {\em one} superspecial abelian variety
with {\em many} principal polarizations.

Let $\msA=(A=E^g,\lambda)$ be a principally polarized superspecial abelian variety of
dimension $g$ over $\oFp$ with $\oFp$-isomorphism class $[\msA]$.
The principal polarization $\lambda$ is an isomorphism from $A$ to
$\hat{A}={\rm Pic}^0(A)$ satisfying the conditions of Definition \ref{radio}.
The number of such isomorphism classes $[\msA]$ is finite and denoted
$h=h_g(p)$; we shall see that $h$ is a type of class number.
Set
\begin{align}
\label{rent}
\sSgpz&=\{\text{$\oFp$-isomorphism classes $[\msA]$}\}\\
\nonumber &=\{[\msA_1],\ldots , [\msA_h]\}\text{ with } \msA_j=(A_j,\lambda_j).
\end{align}
So, for example, 
\begin{align*}
\sS_{\!1}(p)_0&=\{\text{supersingular $j$-invariants in 
characteristic $p$}\}\text{ and}\\
\#\sS_{\!1}(p)_0&=h_1(p)=h(\HH_p), \text {
the class number of the quaternion algebra $\HH_p$.}
\end{align*}

A principal polarization $\lambda$ of an abelian variety $A/\oFp$ defines
the Weil pairing on $A[\ell]$, $\ell\neq p$ prime:
$\langle \,\,\, ,\,\,\,\rangle_{\lambda,\ell}:A[\ell]\times A[\ell]\rightarrow
\mu_\ell$. Put 
\begin{equation}
\label{whole}
\Iso_\ell(\msA)=\{\text{maximal isotropic subgroups 
$C\subseteq A[\ell]$}\};\,\,\,
 \#\Iso_{\ell}(\msA)=\prod_{k=1}^g(\ell^k+1).
\end{equation}

\begin{prp}
{\rm (cf. \cite[\S 23]{m}, \cite[p.~36]{o}).}\,\,
Suppose $\ell\neq p$, $\msA=(A,\lambda)$ is a principally polarized abelian 
variety over $\oFp$, and $C\subseteq A[\ell]$.  Let 
$\psi_C:A\rightarrow A/C\equalscolon A'$.  Then there is a principal
polarization $\lambda'$ on $A'$ so that $\psi_C^\ast \lambda'=\ell \lambda$ 
if and only if $C\in\Iso_\ell(\msA)$.
In this case write $\msA'=(A',\lambda')=\msA/C$.  If $[\msA]\in
\sSgpz$, then $[\msA']\in\sSgpz$.
\end{prp}
\noindent If $[\msA]\in\sSgpz$ and $C\in\Iso_\ell(\msA)$, then
$\msA\rightarrow \msA'=\msA/C$ is an $(\ell)^g$-isogeny.  Such 
$(\ell)^g$-isogenies induce correspondences from the finite set
$\sSgpz$ to itself. These correspondences can be used to define
various graphs---in this paper we define {\em{three}} 
$(\ell)^g$-isogeny graphs:
the big isogeny graph $\GR_{\!g}(\ell,p)$, the little isogeny graph
$\gr_{\!g}(\ell, p)$, and the enhanced isogeny graph $\wgr_{\!g}(\ell,p)$.
In this introduction we content ourselves with defining the
simplest of the three, the big isogeny graph $\GR\colonequals \GR_{\!g}(\ell,p)$:
\begin{dfn}
\label{wine}
The vertices of the graph $\GR=\GR_{\! g}(\ell,p)$ are
$\Ver(\GR)=\sSgpz$, so $h=h_g(p)=\#\Ver(\GR)$.
The (directed) edges of the graph $\GR$ connecting the vertex
$[\msA_i]\in\sSgpz$ to the vertex $[\msA_j]\in\sSgpz$ are
\[
\Ed(\GR)_{ij}=\{C\in\Iso_\ell(\msA_i)\mid [\msA_i/C]=[\msA_j]\}.
\]
\end{dfn}
\noindent The adjacency matrix $\Ad(\GR)_{ij}=\#\Ed(\GR)_{ij}$
is a constant row-sum matrix by \eqref{whole}: 
\begin{equation}
\label{now}
\sum_{j=1}^{h}\#\Ed(\GR)_{ij}=\prod_{k=1}^g(\ell^k+1).
\end{equation}

Yevgeny Zaytman and I spoke at the conference on these isogeny graphs
and our results in \cite{jz}, focusing on the theorem:
\begin{thm}
\label{weg}
\textup{(\cite[\S 8]{jz})}\,\,
The isogeny graphs $\GR_{\!g}(\ell,p)$, $\gr_{\!g}(\ell, p)$, and 
$\wgr_{\!g}(\ell,p)$
are connected.
If $g>1$, the regular graph $\GR_{\!g}(\ell,p)$ is in general not Ramanujan.
\end{thm}
\noindent 
One ingredient of our proof is strong approximation for the quaternionic 
unitary group.  The quaternionic unitary group has previously
been applied to moduli of abelian varieties in characteristic
$p$; see Ekedahl/Oort
\cite[\S 7, esp.~Lemma 7.9]{o1}, Chai/Oort
\cite[Prop.~4.3]{co}, 
and Chai \cite[Prop.~1]{c}.
In my lecture and here
I treat the general background and broader context of the isogeny
graphs.  Zaytman will explain the proof of Theorem \ref{weg}.
A common notation is shared between the two papers.
Full proofs and references for the results considered here can be
found in \cite{jz}.

My task of explaining isogeny graphs in arithmetic
geometry is complicated by their opaque history: The subject 
is certainly over 75 years old, dating back at least to Brandt
\cite{br} from 1943.  During this time, our graphs
appear in disguises and in variations: the big, little, and enhanced
isogeny graphs all are there.  So the broader contexts and work 
done in other settings are not readily accessible.
Let me give a personal example:  The work on these graphs I have used most
from graduate student days to the present is the 1979 paper
\cite{k} of Kurihara entitled {\em On some examples of equations defining
Shimura curves and the Mumford uniformization}. Who would guess that
this had anything to do with isogeny graphs? In fact, the word ``isogeny''
does not appear in the entire paper.

It is perhaps helpful to list in chronological order
the four lives of our isogeny graphs, with \textbf{A}, \textbf{B},
and \textbf{C} subsequently playing a role in our story:\\[.05in]
\textbf{A. 1943 -- : Brandt matrices.}\,
In this first appearance there are no graphs, no elliptic curves,
and no abelian varieties---only the Brandt matrices
which are the adjacency matrices of $\GR_{\!g}(\ell,p)$ and 
the weighted adjacency matrices of $\gr_{\!g}(\ell,p)$.  The 
major theorem was the
trace formula.  Brandt \cite{br} defined the matrices for $g=1$, 
primarily treating definite
quaternion algebras over $\Q$.  Eichler then introduces strong approximation 
and develops the theory for higher weight and totally real fields in case
$g=1$,
including the trace formula.  Shimura \cite{sh}
laid the foundations to generalize to $g>1$ and the quaternionic unitary
group; Brandt matrices in this setting were defined in the 1980's by
Hashimoto, Ibukiyama, Ihara, and Shimizu---see \cite{h}.  
Gross's algebraic modular forms \cite{g} subsequently provided a more
general context for these matrices.\\
\textbf{B. 1976 -- : Shimura curves and $\mathbf{g=1}$.}
In this incarnation the graphs appear, but not from isogenies of
supersingular elliptic curves. Rather they arise from the bad reduction
of Shimura curves in the work of \v{C}erednik and Drinfeld. The explicit
graphs are deduced from the results of \v{C}erednik and Drinfeld in
\cite{jl2}, the jacobian of the  graph $\wgr_{\!1}(\ell, p)$ is computed
in \cite{jl3}, and the integral Hodge theory of the graphs
$\gr_{\!1}(\ell, p)$, $\wgr_{\!1}(\ell,p)$ with applications to congruences
between newforms and old forms is in \cite{jl1}.   
In \S \ref{lion1} of this paper and \cite[\S 9]{jz}
we uniformize  $\gr_{\!1}(\ell,p)$ and 
$\wgr_{\!1}(\ell,p)$ as quotients of the tree $\Delta=\Delta_\ell$ 
for $\SL_2(\Q_\ell)$ using Kurihara
\cite{k}.
The question of how the \v{C}erednik-Drinfeld results generalize
to the higher-dimensional case $g>1$ remains a magnet for research.\\
\textbf{C. 1988 -- : LPS graphs; Ramanujan graphs and complexes.}
In an influential paper, Lubotzky, Phillips, and Sarnak \cite{lps} construct
families of Cayley graphs from the Hamilton quaternions and show that 
they are Ramanujan. This work made the Ramanujan property a central
focus.  These LPS graphs are shown to be explicit
covers of the little isogeny graph 
 $\gr_{\!1}(\ell,2)$ in \cite[\S 3] {jl4}. A higher-dimensional
Ramanujan complex was first constructed in \cite{jl5}.  But the full flowering
of Ramanujan complexes is due to the work over the last 15 years of
Alex Lubotzky and Winnie Li, together with their students, collaborators,
and colleagues.
Obviously this thread begs for a notion of isogeny complex to generalize
isogeny graph; this is the subject of \cite{jz1}.\\
\textbf{D. 2011 -- and 1986 -- : Applications of isogenies.}
In this current optic, applications are found for isogenies
and isogeny graphs.
This begins with Mestre's 1986 ``m\'{e}thode des graphes'' \cite{me}
for computing Hecke operators.  Then another completely different application
is introduced with Jao and de Feo's 2011 proposal  \cite{jd} for an
isogeny-based key exchange using supersingular elliptic curves. 

\section{Polarizations of superspecial abelian varieties}
\label{frog}

Let $X$ be an abelian variety over a 
field $k$ of dimension $g$; the dual abelian variety $\hat{X}=\Pic^{0}(X)$ is defined over $k$.
A homomorphism $\tau:X\rightarrow \hat{X}$ is {\sf symmetric}
if $\hat{\tau}=\tau$, where we identify $X=\hat{\hat{X}}$ via the
canonical isomorphism
\begin{equation}
\label{venus}
\kappa_X:X\stackrel{\simeq}{\longrightarrow}\hat{\hat{X}} \text{ of 
\cite[Thm.~7.9]{vdG}, for example.}
\end{equation}
Let $\mathcal{P}$ be the Poincar\'{e} line bundle on $X\times\hat{X}$.
\begin{dfn}
\label{radio}
\textup{ (cf. \cite[Cor.~11.5, Defn.~11.6]{vdG}) }
A {\sf polarization} of an abelian variety $X$ over a field $k$ is a
symmetric isogeny $\lambda:X\rightarrow \hat{X}$ over $k$ such that the line
bundle $({\rm id}_X,\lambda)^\ast\mathcal{P}$ is ample.
The {\sf degree} $\deg(\lambda)$ of the polarization $\lambda$
is the degree of the isogeny $\lambda$, i.e., $\#\ker(\lambda)$.
The degree $\deg(\lambda)$ is always a square by the Riemann-Roch theorem,
see \cite[\S 16]{m}.
It is convenient to define the {\sf reduced degree} ${\rm rdeg}(\lambda)$
of the polarization $\lambda$ to be ${\rm rdeg}(\lambda)=\sqrt{\deg(\lambda)}$.
A polarization of degree $1$ is a {\sf principal polarization}.
\end{dfn}
\begin{rmk}
\label{mars}
{\rm
For a polarization $\lambda$ we
have 
${\rm rdeg}(n\lambda)=n^{g}{\rm rdeg}(\lambda)$. For an isogeny $f:X\rightarrow X'$
and polarization $\lambda '$ on $X'$, 
${\rm deg} f^\ast(\lambda')={\rm deg}(f)^2\deg \lambda'$.
}
\end{rmk}
\noindent Many of the results in this section can be found in
the paper \cite{iko} by
 Ibukiyama, Katsura, and Oort  and many were known to 
Serre.

Let $\HH$ be a positive definite quaternion algebra over $\Q$ 
with a maximal order $\cO_{\HH}$, main involution $x\mapsto \overline{x}$,
and reduced norm $\Nm_{\HH/\Q}(x)=\Nm(x)=x\overline{x}$.
The reduced norm $\Nm:\Mat_{g\times g}(\HH)\rightarrow \Q$  is the 
multiplicative polynomial of degree $2g$ generalizing the reduced norm
$\Nm:\HH\rightarrow \Q$.  Put
\[
\SL_g(\cO_\HH)=\{ M\in\Mat_{g\times g}(\cO_\HH) \mid \Nm(M)=1\}.
\]

A matrix  $H\in\Mat_{g\times g}(\HH)$ is {\sf hermitian} if 
$H=H^\dagger\colonequals \overline{H}^t$. Set
\begin{equation}
\label{mink}
\msH_g(\cO_\HH)=\{H\in \Mat_{g\times g}(\cO_{\HH})\mid H 
\text{ is positive-definite hermitian}\}.
\end{equation}
The ``Haupt norm'' $\HNm$ of Braun-Koecher \cite[Chap.~2, \S4]{bk}
(see also \cite[Thm.~6 and proof, \S 21]{m}) gives a map
$\HNm:\msH_g(\cO_{\HH})\rightarrow \N$. For an integer $d\geq 1$ put
\begin{equation}
\label{mink1}
\msH_{g,d}(\cO_\HH)=\{H\in \msH_g(\cO_{\HH})\mid 
\HNm(H)=d\}.
\end{equation}

\begin{dfn}
\label{whale}
Let  $\HH$ be a 
definite quaternion division algebra over $\Q$ with  maximal
order $\cO_{\HH}$. Set 
$\cO_{\hat{\HH}}=\cO_\HH\otimes\hat{\Z}$, the profinite completion of 
$\cO_\HH$, and $\hat{\HH}=\cO_{\hat{\HH}}\otimes\Q$.

There is a 
right action of $M\in\SL_g(\cO_\HH)$ on hermitian 
$H\in\msH_{g,d}(\cO_{\HH})$: 
\begin{equation}
\label{grab}
H\cdot M\colonequals M^\dagger H M .
\end{equation}
Set
$\overline{\msH}_{\!\!g,d} (\cO_\HH)\colonequals \msH_{g,d}(\cO_\HH)/\SL_g(\cO_\HH)$ with
$[H]\in\overline{\msH}_{\!\!g,d}(\cO_\HH)$ the class defined by 
$H\in\msH_{g,d}(\cO_\HH)$.

If $B$ is an algebra with anti-involution with fixed ring $R$, set
\begin{align*}
\U_g(B)&=\{M\in\Mat_{g\times g}(B) \mid M^\dagger M=\Id_{g\times g}\},\\
\GU_g(B)&=\{M\in\Mat_{g\times g}(B)\mid M^\dagger M=\lambda \Id_{g\times g}
\text{ with } \lambda\in R^\times\}.
\end{align*}
We will have  $B=\cO=\cO_E,\,\cO_\HH,\,\cO_\HH[1/\ell],\,\HH, \,
\cO_{\hat{\HH}}, \text{ and }\hat{\HH}$ in the course of this paper.
\end{dfn}

We now consider polarizations on $A/\oFp=E^g/\oFp$ with
$\End(E)=\cO_E=\cO\subseteq \HH_p$. 
For $\A=(A, \lambda)$ a polarized superspecial abelian variety over $\oFp$,
let $[\A]$ denote the isomorphism class defined by $\A$ over $\oFp$.

Let $\lambda_0$ be the
standard product polarization of $A$; 
the polarization $\lambda_0$ is principal.  For 
$H\in\msH_{g,d}(\cO)\subseteq
\Mat_{g\times g}(\cO)$, let $\lambda_H$ be the polarization with
\[
\lambda_{H}:A\stackrel{H}{\longrightarrow}A
\stackrel{\lambda_0}
{\longrightarrow} \hat{A}.
\]
If $H\in\msH_{g,d}(\cO)$, then $\lambda_H$ is a polarization of $A=E^g$
with reduced degree ${\rm rdeg}(\lambda_H)=d$.

\begin{thm}
\label{redhawk}
If $g>1$, there are one-to-one correspondences induced by associating the
polarization $\lambda_H$ to the hermitian matrix $H$:
\begin{enumerate}[\upshape (a)]
\item
\label{redhawk1}
polarizations $\lambda$ of $A=E^g$  with ${\rm rdeg}(\lambda)=d
\longleftrightarrow \msH_{g,d}(\cO)$ and
\item
isomorphism classes $[\A=(A=E^g, \lambda)]$ with
${\rm rdeg}(\lambda)=d \longleftrightarrow 
\overline{\msH}_{\!\!g,d}(\cO)$.
\end{enumerate}
\end{thm}

For $H\in\msH_{g,d}(\cO)$ with $g>1$, we denote by 
$\A(H)=(A=E^g,\lambda_H)$ the associated
polarized superspecial abelian variety.
Theorem \ref{redhawk} allows us to describe the isomorphism
classes $\sSgpz$ of principally polarized superspecial abelian
varieties $[\msA]$ over $\oFp$ as in \eqref{rent}.

\begin{prp}
\label{radish}
If $g>1$, there is a bijection
\[
\overline{\msH}_{\!\!g,1}(\cO)\leftrightarrow \sSgpz
\]
associating $[H]\in\overline{\msH}_{\!\!g,1}(\cO)$ to $[\A(H)]\in\sSgpz$.
\end{prp}

As discussed in $\S 1$, the theory of superspecial abelian varieties
bifurcates into the cases  $g>1$ and $g=1$.  In spite of this we are
able to give a uniform treatment of the principally polarized
isomorphism classes in Theorem \ref{toad} below by shifting from the hermitian matrices for 
$g>1$ in Proposition \ref{radish} to a notion of hermitian modules.

Let 
$H_0$ be the hermitian form on $\HH^g$ given by $\Id_{g\times g}$.
Let $L\subseteq \HH^g$ be a finitely generated
 right $\cO_\HH$-module such that $L\otimes\Q=\HH^g$. We say that 
$L$ is {\sf principally polarized} if there exists 
$c\in\Q^\times$ such that  $cH_0|_L$
is $\cO_\HH$-valued and unimodular. We define the {\sf dual}
of $L$ to be $\hat{L}=c^{-1}L$.

We can classify principally polarized right $\cO_\HH$-modules.
For $M\in\GU_g(\hat{\HH})$, denote by $[M]$ the coset containing $M$
in $\GU_g(\hat{\HH})/\GU_g(\cO_{\hat{\HH}})$ and define the
principally polarized right $\cO_{\HH}$-module $\gamma(M)\colonequals
M\cO^g_{\hat{\HH}}\cap \HH^g$. 

\begin{thm}
A one-to-one correspondence
\[
\{\text{principally polarized   right $\cO_\HH$-modules}\}\leftrightarrow
\GU_g(\hat{\HH})/\GU_g(\cO_{\hat{\HH}})
\]
is given by 
$\GU_g(\hat{\HH})/\GU_g(\cO_{\hat{\HH}})\ni M\leftrightarrow
\gamma(M)=M\cO^g_{\hat{\HH}}\cap \HH^g$.
\end{thm}

Finally we define the
{\sf classes} of $\GU_g(\cO_\HH)$, denoted $\msP_g(\cO_\HH)$, as the
equivalence classes $[L]$ of principally polarized right
$\text{$\cO_\HH$-submodules $L$ up to left multiplication by 
$\GU_g(\HH)$}$:
\begin{equation}
\label{mallard}
\msP_g(\cO_\HH)=\GU_g(\HH)\backslash \GU_g(\hat{\HH})/\GU_g(\cO_{\hat{\HH}}).
\end{equation}
The set $\msP_g(\cO_\HH)$ is finite. We define the {\sf class number} $h_g(\HH)$
of $\GU_g(\cO_{\HH})$ by 
\begin{equation}
\label{fish}
h_g(\HH)=\#\msP_g(\cO_\HH);
\end{equation}
it is independent of the choice of maximal order $\cO_\HH$.
\begin{rmk}
\label{cow}
{\rm
In case $g=1$, \eqref{mallard} becomes
$\msP_1(\cO_\HH)=\HH^\times\backslash\hat{\HH}^\times/
\cO_{\hat{\HH}}$.  Hence $\msP_1(\cO_\HH)$ is the usual ideal classes
of $\HH$ and $h_1(\HH)=\#\msP_1(\cO_\HH)$ is the usual class number
$h(\HH)$ of the quaternion algebra $\HH$. In particular
$\msP_1(\cO)$ is in one-to-one correspondence with
$\sS_1(p)_0$.
}
\end{rmk}
For $g>1$ we can relate principally polarized right $\cO_\HH$-modules
to hermitian matrices using strong 
approximation---see \cite[\S 2]{jz}---obtaining:
\begin{thm}
\label{thrush}
If $g>1$, the set $\overline{\msH}_{\!\!g,1}(\cO_{\HH})$ is in one-to-one 
correspondence with $\msP_g(\cO_\HH)$.
\end{thm}

We thus obtain the following description of $\sSgpz$.
\begin{thm}
\label{toad}
We have one-to-one correspondences:
\begin{enumerate}[\upshape (a)]
\item
For $g\geq 1$, $
\sSgpz\longleftrightarrow \msP_g(\cO)= \GU_g(\HH_p)\backslash 
\GU_g(\hat{\HH}_p)/\GU_g(\cO_{\hat{\HH}_p})$.
\item
For $g>1$,
$\sSgpz\longleftrightarrow \msP_g(\cO)\longleftrightarrow
\overline{\msH}_{\!\!g,1}(\cO)$ .
\end{enumerate}
\end{thm}
\noindent In particular, with $h_g(p):=\#\sS_{\!g}(p)_0$ as in \S 1, we 
have $h_g(p)=h_g(\HH_p)$.

\section{Brandt matrices}
\label{dog}

Let $h=h_g(\HH)$ and $\msP_g(\cO_\HH)=\{[L_1],\ldots ,[L_h]\}$
with $[L_i]$ the class defined by the principally polarized right
$\cO_{\HH}$-module $L_i$ as in \eqref{mallard}.
For $1\leq j\leq h$, set 
\[
e_{g}(j)=\#\{U\in\GU_g(\HH) \mid L_j=UL_j\}.
\]
\begin{dfn}
\label{goose}
For $n\geq 1$ define the {\em Brandt matrix} $B_g(n)\in
\Mat_{h\times h}(\Z)$ by
\[
B_g(n)_{ij}=\frac{\#\{U\in\GU_g(\HH)\mid [L_i:UL_j]=n^{2g}\}}{e_g(j)}
\]
and define $B_g(0)_{ij}=1/e_g(j)$.
\end{dfn}

Suppose $g=1$.  
The class number $h_1(\HH)$ in \eqref{fish} is 
the usual class number $h=h(\HH)$ of
the quaternion algebra $\HH$ by Remark \ref{cow}.
The principally polarized right $\cO_\HH$-modules
$L_1$, $L_2$, \ldots, $L_h$ can be identified with representatives
 $I_1$, $I_2$, \ldots, $I_h$
for the right $\cO_\HH$-ideal classes.
The norm $\Nm(I)$ of a (right or left) fractional $\cO_\HH$-ideal 
is the positive
rational number generating the fractional ideal of $\Q$ generated
by $\{\Nm(\alpha)\mid \alpha\in I\}$.
Let $\cO_i$ be the left order of the right $\cO_\HH$-ideal $I_i$.
Then $e(i)=e_1(i)=\#\cO_i^\times$.  We thus have
\[
B(n)_{ij}=B_1(n)_{ij}=\frac{\#\{\lambda\in I_iI_j^{-1}\mid \Nm(\lambda)=
n\Nm(I_iI_j^{-1})\}}{e(j)}
\]
and $B(0)_{ij}=1/e(j)$, which
is precisely the classical definition of Brandt matrices for
a rational definite quaternion algebra.

The Brandt matrices $B_g(n)$ for a maximal order $\cO_\HH\subseteq\HH$
are amenable to machine computation, although the memory requirements
rapidly grow with $n$ and especially $g$ so that 
few examples are accessible
with $g=3$.  We had no computations finish for $g\geq 4$.
As an example, 
take $\HH=\HH_5$, the rational definite quaternion algebra of discriminant $5$.
The first class numbers of $\HH_5$ are: $h_1(\HH_5)=1$,
$h_2(\HH_5 )=2$, $h_3(\HH_5)=3$. The Brandt matrix $B_g(\ell)$
has constant row-sum $\prod_{k=1}^g(1+\ell^k)$.
Brandt matrices for $\HH_5$ with $g=1,\,2,\,3$ are given
in Table 1, where ? means the computation did not finish.
\begin{center}
\begin{table}[h]
\label{foggy1}
\begin{tabular}{l c c c c}
   &  $\bm{B_g(2)}$ &  $\bm{B_g(3)}$ & 
$\bm{B_g(7)}$ &  $\bm{B_g(11)}$\\[.15in]
$\bm{g=1}$\hspace*{.2in} & [3] & [4] & [8] & [12] \\[.2in]
$\bm{g=2}$ & $\begin{bmatrix} \,12\, & \,3 \,\\ \,10\, & \,5\,
\end{bmatrix}$ &
 $\begin{bmatrix} \,34\, & \,6\, \\ \,20\, & \,20\,\end{bmatrix}$ &
\!\quad$\begin{bmatrix} \,322\, & \,78\, \\ \,260\, & \,140\,\end{bmatrix}$
\!\quad &
\!\quad$\begin{bmatrix} \,1164\, & \,300\, \\ \,1000\, & \,464\,\end{bmatrix}$
\!\quad\\[.3in]
$\bm{g=3}$ &  $\!\!\!\!\quad \begin{bmatrix} \,54\, & \,27\, & \,54\, \\ \,30\, 
& \,15\, & \,90\,\\
\,14\, & \,21\, & \,100\,\end{bmatrix}\!\!\!\!\quad $ &
$\!\!\!\!\quad\begin{bmatrix} \,292\, & \,180\, & \,648\, \\ \,200\, & 
\,200\, & \,720\,\\
\,168\, & \,168\, & \,784\,\end{bmatrix}\!\!\!\!\quad$ & \text{\Large ?} &   \text{\Large ?}
\end{tabular}\\[.25in]
\caption{\label{bt}Brandt matrices $B_g(\ell)$ for $\HH_5$}
\end{table}
\end{center}

\section{The big, little, and enhanced isogeny graphs}
\label{cat}

We will consider $(\ell)^g$-isogenies of principally polarized
superspecial abelian varieties in charteristic $p$ with $p\neq \ell$; see 
Section 1 for the definitions.  As discussed there,
there are {\em three}
natural graphs constructed from superspecial
abelian variety isogenies---the big isogeny graph $\GR_{\!g}(\ell,p)$,
the little isogeny graph $\gr_{\!g}(\ell,p)$, and the enhanced
isogeny graph $\wgr_{\!g}(\ell,p)$.
The different graphs arise depending on how isogenies and
polarizations are identified.  
Big, little, and enhanced isogeny graphs have subtly different
properties, so we need to be careful with the definitions...

\begin{dfn}
\label{roar}
{\rm 
 A {\sf graph}
 $\Grr$ has a set of  vertices $\Ver(\Grr)=\{v_1,\ldots,v_s\}$
and a set of (directed) edges
$\Ed(\Grr)$. And edge $e\in\Ed(\Grr)$ has initial vertex $o(e)$ and
terminal vertex $t(e)$.
For $v_i, v_j\in\Ver(\Grr)$, put
\[
\Ed(\Grr)_{ij}=\{e\in\Ed(\Grr)\mid o(e)=v_i\text{ and }
t(e)=v_j\}.
\]
The
{\sf adjacency matrix} 
$\Ad(\Grr)\in\Mat_{s\times s}(\Z)$ of $\Grr$ is defined as
\[
\Ad(\Grr)_{ij}=\#
\Ed(\Grr)_{ij}.
\]

We place no further restrictions on our definition of a graph.
Serre \cite{s} requires graphs to be {\sf graphs with opposites}:
every directed edge $e\in\Ed(\Grr)$ has an {\sf opposite}
edge $\overline{e}\in\Ed(\Grr)$ with $\overline{\overline{e}}=e$.
An edge $e$ with $\overline{e}=e$
is called a {\sf half-edge}. Serre forbids half-edges; we will call a graph
satisfying his requirements a {\sf graph without half-edges}.
Kurihara \cite{k} relaxes Serre's definition to allow
half-edges giving the notion of a {\sf graph with half-edges}.
(A graph with half-edges may have $\emptyset$ as its set of half-edges,
so every graph without half-edges is a graph with half-edges.)
Following  \cite{k}, if $\Grr$ is a graph with half-edges, $\Grr^\ast$
is the graph with the half-edges removed.

A {\sf graph with weights}  is a graph $\Grr$ with opposites together with
 a weight function
$\w:\Ver(\Grr)\cup\Ed(\Grr)\rightarrow \N$ 
satisfying $\w(e)=\w(\overline{e})$
and $\w(e)|\w(o(e))$ for each edge $e$. The
{\sf weighted adjacency matrix} $\Adw(\Grr)$ of a graph with weights
 $\Grr$ is
\begin{equation}
\label{ritual}
\Adw(\Grr)_{ij}=\sum_{e\in\Ed(\Grr)_{ij}}\frac{\w(v_i)}{\w(e)}.
\end{equation}
Following \cite[\S 3]{k}, 
a {\sf graph with lengths} is a graph $\Grr$ with opposites together
with a length function $f:\Ed(\Grr)\rightarrow \N$ satisfying
$f(e)=f(\overline{e})$ for $e\in\Ed(\Grr)$. A graph with weights
determines a graph with lengths by taking the length of an edge to
be its weight.  If $\Grr$ is a graph with weights or lengths, then
$\Grr^\ast$ inherits weights or lengths, respectively, from $\Grr$.
}
\end{dfn}
\subsection{The big isogeny graph $\GR\colonequals \GR_{\!g}(\ell,p)$}
\label{chick}

The big isogeny graph $\GR=\GR_{\!g}(\ell,p)$ was defined in 
Definition \ref{wine}; this is the usual
``isogeny graph''.   In particular, 
$\Ver(\Gr)=\sSgpz$, so $\#\Ver(\Gr)=h=h_g(p)$. We have
\[
\Ed(\Gr)_{ij}=\{C\in\Iso_\ell(\msA_i)\mid [\msA_i/C]=[\msA_j]\} 
\]
with $\Iso_\ell(\msA)$ as in \eqref{whole}.
The adjacency matrix
$\Ad(\Gr)$ is a constant row-sum matrix as in \eqref{now}:
\[
\sum_{j=1}^h\Ad(\Gr)_{ij}=\prod_{k=1}^g(\ell^k+1) .
\]
It is in fact a familiar matrix:
\begin{thm}
Let $B_g(\ell)$ be the Brandt matrix for $\cO\subseteq \HH_p$.
Then $\Ad(\Gr_{\!g}(\ell,p))=B_g(\ell)$.
\end{thm}

The adjacency matrix $\Ad(\GR)=B_g(\ell)$ is not in general
symmetric, so $\GR$ cannot be a graph with opposites.  In particular, 
taking the dual isogeny does {\em not} give a well-defined involution
on $\Ed(\Gr)$, so $\Gr$ is {\em not} a graph with opposites via dual 
isogenies.

\subsection{The little isogeny graph $\gr\colonequals\gr_{\!g}(\ell,p)$}
\label{sunny}

The little $(\ell)^g$-isogeny graph $\gr=\gr_{\!g}(\ell,p)$
has vertices
\[
\Ver(\gr)=\sSgpz, 
\]
so $\Ver(\gr)=\Ver(\GR)$ and 
$\#\Ver(\gr)=h=h_g(p)$.
If $[\msA]\in\sS_g(p)_0$ and $C,C'\in\Iso_\ell(\msA)$,
say $C\sim C'$ if there exists $\alpha\in\Aut(\msA)$ such that
$\alpha C=C'$. The class $[C]\in\iso_\ell(\msA)\colonequals\Iso_\ell(\msA)/\sim$ is defined by $C\in\Iso_\ell(\msA)$.  We put
\[
\Ed(\gr)_{ij}=\{[C]\in \iso_\ell(\msA_i)  \mid [\msA_i/C]=[\msA_j]\}.
\]

Unlike the big isogeny graph, the little isogeny graph $\gr$
{\em is} a graph with opposites:
the dual isogeny gives a well-defined involution on $\Ed(\gr)$.
In general we have edges $e\in\Ed(\gr)$ with $\overline{e}=e$, so
$\gr$ is a graph with half-edges.  Beyond this, 
$\gr$ is a graph with weights: set $\w([\msA])=\#\Aut(\msA)$ and
$\w([C])=\#\Aut(A,\lambda,C)$ for a vertex corresponding to 
$[\msA=(A,\lambda)]\in
\sS_g(p)_0$ and the edge emanating from that vertex corresponding
to $[C]\in\iso_\ell(\msA)$. The weighted adjacency matrix \eqref{ritual}
 of the
little isogeny graph is the Brandt matrix:
\begin{thm}
$\Adw(\gr_{\!g}(\ell,p))=\Ad(\Gr_{\!g}(\ell,p))=B_g(\ell)$.
\end{thm}

\subsection{The enhanced isogeny graph $\wgr\colonequals\wgr_{\!g}(\ell,p)$}

Recall the notation \eqref{rent}:
\[
\sSgpz=\{[\msA_1],\ldots , [\msA_h]\}\equalscolon\{v_1,\ldots , v_h\}.
\]
Suppose $[\msA=(A,\lambda)]\in\sSgpz$. Let 
$\ell\msA\colonequals (A, \ell \lambda)$, a $g$-dimensional superspecial
abelian variety with $\ell$ times a principal polarization (which
we call an $[\ell]$-polarization of type $g$ in \cite{jz}). Set
\[
\sSgpg = \{[\ell\msA_1],\ldots, [\ell \msA_h]\}\equalscolon
\{v_{h+1},\ldots, v_{2h}\}.
\]
Define
the $[\ell]$-dual $\hat{\msA}=(\hat{A},[\lambda])$ of
$[\msA]\in\sSgpz\coprod\sSgpg$ by requiring that the composition
\[
[\lambda]\circ \lambda:A\stackrel{\lambda}{\longrightarrow}
\hat{A}\stackrel{[\lambda]}
{\longrightarrow}A
\]
from $A$ to itself is multiplication by $\ell$.  This $[\ell]$-dual
construction interchanges type $0$ (principal polarizations) and 
type $g$: If 
$[\msA]\in\sSgpz$, then $[\hat{\msA}]\in\sSgpg$; and if
$[\msA]\in\sSgpg$, then $ [\hat{\msA}]\in\sSgpz$.

We can now define the enhanced isogeny graph
$\wgr\colonequals \wgr_{\!g}(\ell, p)$:
\begin{dfn}
\label{halibut}
{\rm
The vertices of $\wgr=\wgr_{\!g}(\ell,p)$ are
\[
\Ver(\wgr)=\sSgpz\coprod \sSgpg =\{v_1,\ldots , v_h\}\coprod \{v_{h+1}, \ldots
, v_{2h}\}.
\] 
Hence $\#\Ver(\wgr)=2h=2h_g(p)$.

The edges connecting the vertex $v_{h+i}=[\ell\msA_i]\in\sSgpg$ 
to the vertex $v_j=[\msA_j]\in\sSgpz$ are
\[
\Ed(\wgr)_{h+i,j}=\{[C]\in\iso_\ell(\msA_i) \mid [\msA_i/C]=[\msA_j]\}
\]
with $\iso_\ell(\msA)$ as in \S \ref{sunny}.
For $v_i=[\msA_i]\in\sSgpz$ and $v_{h+j}=[\ell\msA_j]\in\sSgpg$,
\[
\Ed(\wgr)_{i,h+j}=\{[\hat{C}]\in\iso_\ell(\hat{\msA}_i) 
\mid [\hat{\msA}_i/\hat{C}]=[\hat{\msA}_j]\}.
\]
In case $1\leq i,j\leq h$ or $h+1\leq i,j\leq 2h$,  $\Ed(\wgr)_{ij}=\emptyset$.

The enhanced isogeny graph $\wgr$ is a graph with opposites:
If $e\in\Ed(\wgr)_{ij}$,
the opposite edge $\overline{e}\in\Ed(\wgr)_{ji}$ is the equivalence
class of the dual isogeny. We never have $\overline{e}=e$, so $\wgr$ is 
a graph without half-edges.  The graph $\wgr$ is a graph with weights:
define $\w$ as the
order of automorphism group as for $\gr$.
}
\end{dfn}

\begin{thm}{\rm (a) }
The enhanced isogeny graph $\wgr$ is the bipartite double cover of
the little isogeny graph $\gr$ with inherited weights.\\[.08in]
{\rm (b)}
Let $\mathit{Ad}=\Ad(\gr)$ and $\mathit{Ad}_w =\Adw(\gr)=\Ad(\Gr)$. Then
\[
\Ad(\wgr)=\begin{bmatrix}\,0\, &\, \mathit{Ad}\,\\ 
\mathit{Ad} & 0\end{bmatrix}\,\text{ and }
\, \Adw(\wgr)=\begin{bmatrix} 0 & \mathit{Ad}_w \\ \mathit{Ad}_w
 &0\end{bmatrix}=
\begin{bmatrix} 0 & B_g(\ell)\\ B_g(\ell) & 0\end{bmatrix}.
\]
\end{thm}
Let $\iota:\wgr\rightarrow \wgr$ be the involution defined 
on vertices by $\iota([\msA])=[\hat{\msA}]$ and on edges such that if
$e\in\Ed(\wgr)_{ij}$ corresponds to the class $[C]$, then
$\iota(e)\in\Ed(\wgr)_{i+h, j+h}$ (where the indices are added 
$\bmod 2h$) also corresponds to the class $[C]$.
Then $\iota$ fixes no vertices and no
edges of $\wgr$ and $\wgr/\iota=\gr$.

\section{$\ell$-adic uniformization of isogeny graphs}
\label{lion}

In this section we give the uniformization of the isogeny graphs
$\gr_1(\ell,p)$ and $\wgr_1(\ell, p)$ by the Bruhat-Tits building
$\Delta=\Delta_\ell$ of $\SL_2(\Q_\ell)$, which is an $(\ell +1)$-regular
tree.  We then use this uniformization to relate 
$\gr_1(\ell ,p)$ and $\wgr_1(\ell, p)$ to the bad reduction of 
Shimura curves.

\subsection{$\ell$-adic uniformization of isogeny graphs for  $g=1$}
\label{lion1}

The quaternion algebra 
$\HH(p)\colonequals \HH_p$ is split at the prime $\ell$, so 
\[
\Gamma_0\colonequals \cO[1/\ell]^\times\hookrightarrow
\GL_2(\Q_\ell)\cong (\HH(p)\otimes_{\Q}\Q_\ell)^\times.
\]
Set
$\Gamma_1\colonequals \{\gamma\in\Gamma_0 \mid
\Nm_{\HH(p)/\Q}(\gamma)=1 \}$.  Then $\Gamma_1\backslash\Delta$
and $\Gamma_0\backslash \Delta$ are finite graphs with weights
defined by the orders of the stabilizer subgroups for the action on
$\Delta$.
Kurihara \cite{k} shows the following:
\begin{thm}[Kurihara]
\label{dream}
Let $B_1(\ell)$ the Brandt matrix at $\ell$ for the
maximal order $\cO\subseteq\HH(p)=\HH_p$.
\begin{enumerate}[\upshape (a)]
\item
\label{dream1}
$\Ad_{\w}(\Gamma_{0}\backslash\Delta)=B_1(\ell)$.
\item
\label{dream2}
The graph with weights 
$\Gamma_1\backslash\Delta$ is
the bipartite double cover of  $\Gamma_0\backslash
\Delta$.
\end{enumerate}
\end{thm}
\begin{proof}
\eqref{dream1}: \cite[p.~294]{k}.\\
\eqref{dream2}: \cite[p.~296]{k}.
\end{proof}
In \cite[\S 9]{jz}, we prove the following:
\begin{thm}
\label{dock}
\begin{enumerate}[\upshape (a)]
\item
\label{dock1}
$\gr_{\!1}(\ell,p)\cong\Gamma_{0}\backslash \Delta$ as graphs with weights.
\item
\label{dock2}
$\wgr_{\!1}(\ell,p)\cong\Gamma_1\backslash \Delta$ as graphs with weights.
\end{enumerate}
\end{thm}
\noindent Note that 
the big isogeny graph $\Gr_{1}(\ell,p)$ is {\em not}
$\ell$-adically uniformized since, as we saw in Section 4.1, it is 
not even a graph with opposites.

Theorem \ref{dock} in turn will show that our isogeny graphs $\gr_{\!1}(\ell,p)$
and $\wgr_{\!1}(\ell,p)$ arise from the bad reduction of Shimura curves,
which we now explain. Let 
$B$ be the indefinite rational quaternion division algebra with
$\Disc B=\ell p$. Let 
$V_B/\Q$ be the Shimura curve parametrizing principally polarized abelian
surfaces with QM (quaternionic multiplication)
 by a maximal order $\mathcal{M}\subseteq B$.  There
is then a model $M_B/\Z$ of $V_B/\Q$ constructed as a coarse  moduli
scheme by Drinfeld \cite{d}; see also \cite{jl2}. Let 
$\msL/\Z_\ell$ be the $\ell$-adic upper half-plane.
The dual graph $G(\msL/\Z_\ell)$
of its
special fiber is canonically $\Delta=\Delta_\ell$.  For $\Gamma\subseteq
\PGL_2(\Q_\ell)$ a discrete, cocompact subgroup,
the quotient $\Gamma\backslash\msL$
is the formal completion of a scheme $\msL_\Gamma/\Z_\ell$
along its closed fiber. We have that $\msL_\Gamma/\Z_\ell$ is an
{\em admissible curve} in the sense of \cite[Defn.~3.1]{jl2}. Its  dual
graph $G(\msL_\Gamma/\Z_\ell)$ as in \cite[Defn.~3.2]{jl2} is a graph with
lengths and 
$G(\msL_\Gamma/\Z_\ell)
\simeq (\Gamma\backslash\Delta)^\ast$, see \cite[Prop.~3.2]{k}.

For the formulation below, see \cite{jl2}.
\begin{thm}[\v{C}erednik, Drinfeld]
\label{font}
Let $w_\ell$ be the Atkin-Lehner involution at $\ell$ of $M_B$.
Let $\overline{\Gamma}_0$ be the image of $\Gamma_0\subseteq \GL_2(\Q_\ell)$
in $\PGL_2(\Q_\ell)$ and similarly for $\overline{\Gamma}_1$.
Let $\mathfrak{O}$ be the ring of integers in the unramified quadratic
extension of $\Q_\ell$.
\begin{enumerate}[\upshape (a)]
\item
The scheme $M_B\times\Z_{\ell}$ is the twist of $\msL_{\overline{\Gamma}_1}/\Z_\ell$
given by the $1$-cocycle 
\begin{align*}
\chi\in H^1(\Gal(\mathfrak{O}/\Z_\ell)&,
\Aut(\msL_{\overline{\Gamma}_1}\times\mathfrak{O}/\mathfrak{O})),
\text{ where } \chi:\Frobb_\ell\mapsto w_\ell:\\
&\quad M_B\times \Z_\ell=(\msL_{\overline{\Gamma}_1})^\chi .
\end{align*}
\item
$(M_B/w_\ell)\times\Z_{\ell}=\msL_{\overline{\Gamma}_0}/\Z_\ell$.
\end{enumerate}
\end{thm}
\begin{crl}
\begin{enumerate}[\upshape (a)]
\item
$G(M_B\times \Z_\ell)=\Gamma_1\backslash \Delta=\wgr_{\!1}(\ell,p)$
as graphs with lengths.
\item
$G((M_B/w_\ell)\times\Z_\ell)=(\Gamma_{0}\backslash\Delta)^\ast=
\gr_{\!1}(\ell,p)^\ast$ as graphs with lengths 
with $(\Gamma_{0}\backslash\Delta)^\ast$,
$\gr_{\!1}(\ell,p)^\ast$ as in Definition \textup{\ref{roar}}.
\end{enumerate}
\end{crl}

\subsection{$\ell$-adic uniformization of isogeny graphs for  $g>1$}

We would like to generalize Theorem \ref{dock} to $g\geq 1$.
Recall that $A=E^g$, $\cO=\End(E)$, and $\End(A)=\Mat_{g\times g}(\cO)$.
Let  $\cB_g$ be the Bruhat-Tits building for $\Sp_{2g}(\Q_\ell)$ with
$\cSS_{\!g}$ its special $1$-skeleton of $\cB_g$: the vertices
of $\cSS_{\!g}$ are the special vertices of $\cB_g$ and its edges
are the edges of $\cB_g$ between special vertices.

We prove the following theorem in \cite[\S 9]{jz}.

\begin{thm}
The groups $\GU_g(\cO[1/\ell])$, $\U_g(\cO[1/\ell])$
are as in Definition \textup{\ref{whale}}.
\begin{enumerate}[\upshape (a)]
\item
$\gr_{\!g}(\ell,p)=\GU_g(\cO[1/\ell])\backslash \cSS_{\!g}$ as graphs with weights.
\item
$\wgr_{\!g}(\ell,p)= \U_g(\cO[1/\ell])\backslash\cSS_{\!g}$ as graphs with weights.
\end{enumerate}
\end{thm}
In case $g=1$, $\Sp_2(\Q_\ell)=\SL_2(\Q_\ell)$,
$\cSS_{1}=\Delta_\ell$, $\U_1(\cO[1/\ell])=\Gamma_1$,
$\GU_1(\cO[1/\ell])=\Gamma_0$, and we recover Theorem \ref{dock}:
$\gr_{\!1}(\ell,p)=\Gamma_0\backslash \Delta_\ell$,
$\wgr_{\!1}(\ell,p)=\Gamma_1\backslash \Delta_\ell$.
As remarked in Section 1, there is great interest in
generalizing Theorem \ref{font} to $g>1$.


\end{document}